\documentclass[12pt]{amsart}
\usepackage[margin=1in]{geometry}
\usepackage{hyperref,amssymb,amsfonts,amsthm,mathtools,enumerate,verbatim,
color}
\usepackage[alphabetic]{amsrefs}

\newtheorem{Theorem}{Theorem}[section]
\newtheorem{Proposition}[Theorem]{Proposition}
\newtheorem{Lemma}[Theorem]{Lemma}
\newtheorem{Corollary}[Theorem]{Corollary}
\newtheorem*{Claim}{Claim}

\theoremstyle{definition}
\newtheorem{Definition}[Theorem]{Definition}
\newtheorem{Fact}[Theorem]{Fact}

\newtheorem*{Question}{Question}

\newcommand{\efa}{\mathsf{EFA}}
\newcommand{\rca}{\mathsf{RCA}_0}
\newcommand{\wkl}{\mathsf{WKL}_0}
\newcommand{\aca}{\mathsf{ACA}_0}
\newcommand{\rt}{\mathsf{RT}}
\newcommand{\bl}{\mathsf{BL}}

\newcommand{\bst}{\mathsf{B}\Sigma^0_2}
\newcommand{\bsn}{\mathsf{B}\Sigma^0_n}
\newcommand{\ist}{\mathsf{I}\Sigma^0_2}

\newcommand{\bs}[1]{\mathsf{B}\Sigma_ {#1}}
\newcommand{\is}[1]{\mathsf{I}\Sigma_ {#1}}

\DeclareMathOperator{\ran}{\mathrm{ran}}

\newcommand{\floor}[1]{\lfloor #1 \rfloor}

\newcommand{\gs}{gs}

\newcommand{\imp}{\rightarrow}

\newcommand{\biimp}{\leftrightarrow}

\newcommand{\Nb}{\mathbb{N}}

\newcommand{\Sb}{\mathbb{S}}

\newcommand{\Ps}{\mathcal{P}}

\title{Brown's Lemma in second-order arithmetic}

\author[Frittaion]{Emanuele Frittaion}
 \address{Mathematical Institute, Tohoku University, Japan}
\thanks{Emanuele Frittaion's research is supported by the Japan Society for the
Promotion of Science.}
\email{frittaion@math.tohoku.ac.jp}
\urladdr{http://www.math.tohoku.ac.jp/~frittaion/}


\begin{document}


\begin{abstract}
Brown's lemma states that in every finite coloring of the natural numbers there is a homogeneous piecewise syndetic set.
We show that Brown's lemma  is equivalent to
$\ist$ over $\rca^*$. We show in contrast that (the infinite) van der Waerden's theorem is equivalent to $\bst$ over $\rca^*$. We finally consider the finite version of Brown's lemma and show that it is provable in $\rca$ but not in $\rca^*$.  
\end{abstract}

\maketitle

\section{Introduction}

In the present paper we study Brown's lemma in the context of second-order arithmetic. Brown's lemma asserts that piecewise syndetic sets are large, in the sense that whenever we partition the natural numbers into finitely many sets, at least one set must be piecewise syndetic. 

\begin{Definition}[Piecewise syndetic]
A set $X\subseteq \Nb$ is \emph{piecewise syndetic}  if there exists 
$d\in\Nb$ 
such 
that $X$ contains arbitrarily large finite sets with gaps bounded by $d$, where a set has 
gaps bounded by $d$ if the difference between any two consecutive 
elements is $\leq d$.  We also say 
that $X$ is piecewise $d$-syndetic if $d$ witnesses that $X$ is piecewise syndetic. 
\end{Definition}

\begin{Theorem}[Brown's Lemma]\label{brown}
Every finite coloring $C\colon \Nb\to r$ has a $C$-homogeneous piecewise syndetic set. 
\end{Theorem}

Note that piecewise syndetic sets are closed under supersets and so we can reformulate Brown's lemma as follows.

\begin{Theorem}[Brown's lemma, restated]
Let $\Nb=X_0\cup\ldots X_{r-1}$  be a finite partition of the natural numbers. Then $X_i$ is piecewise syndetic for some $i<r$.
\end{Theorem}

In fact, Brown's lemma generalizes to finite partitions of piecewise syndetic sets. Recall that a class of sets  $\Sb\subseteq\Ps(\Nb)$ is called \emph{partition regular} if for every $X\in \Sb$ and for every finite partition $X=X_0\cup\ldots X_{r-1}$ there exists $i<r$ such that $X_i\in\Sb$.

\begin{Theorem}[Partition regularity of piecewise syndetic sets]\label{partition PS}
Let $X\subseteq\Nb$ be piecewise syndetic and $X=X_0\cup\ldots X_{r-1}$  be a finite partition. Then $X_i$ is piecewise syndetic for some $i<r$.
\end{Theorem}

Theorem \ref{brown} was originally proved by Brown \cite{Brown68} (see also \cite{Brown69, Brown71})  in the context of locally finite semigroups.
An ergodic-theoretic proof of Theorem \ref{brown} as well as Theorem \ref{partition PS} can be found in Furstenberg  \cite[Theorem 1.23, Theorem 1.24]{Furst81}. For an algebraic-theoretic proof  of Theorem \ref{partition PS} based on the characterization of piecewise syndetic sets  in terms of $\beta\Nb$, the Stone-\v{C}ech compactification of $\Nb$,  see Hindman \cite{Hind12}. 

As piecewise syndetic sets are arithmetically ($\Sigma^0_3$)-definable and closed under 
supersets, Brown's lemma does not actually assert the existence of a set. 
In other words, this is a $\Pi^1_1$-statement.  Notice that a partition lemma of the form
\begin{itemize}
 \item[$\dag$] Every finite coloring $C\colon\Nb\to r$ has a $C$-homogeneous 
\emph{large} set,
\end{itemize}
where \emph{large} is a property about sets closed 
under supersets, is computably true. If \emph{large} is also arithmetical, then  $\dag$ is likely to be provable in $\rca+\is{n}$ if \emph{large} is $\Sigma^0_{n+1}$-definable and in $\rca+\bsn$ if \emph{large}  is $\Pi^0_n$-definable. For instance the infinite pigeonhole 
principle 
\begin{itemize}
 \item[$\rt^1_{<\infty}$]  Every finite coloring $C\colon\Nb\to r$ has an infinite $C$-homogeneous 
set,
\end{itemize}
is provable and actually equivalent to $\bst$ over $\rca$. Therefore $\rt^1_{<\infty}$ and Brown's lemma serve as an example for $n=2$.  \smallskip

Brown's lemma is related to the well-known van der Waerden's theorem.

\begin{Theorem}[Van Der Waerden's Theorem]
Every finite coloring $C\colon \Nb\to r$ has a $C$-homogenous set with arbitrarily long 
arithmetic progressions. 
\end{Theorem}

Say that $X\subseteq\Nb$ is AP if it contains arbitrarily long arithmetic 
progressions. Note that AP sets are also closed under supersets and so we can restate van der Waerden's theorem as follows.

\begin{Theorem}[Van Der Waerden's Theorem, restated]
Let $\Nb=X_0\cup\ldots X_{r-1}$  be a finite partition of the natural numbers. Then $X_i$ is AP for some $i<r$.
\end{Theorem}

Also van der Waerden's theorem generalizes to finite partitions of AP sets. 

\begin{Theorem}[Partition regularity of AP sets]\label{partition AP}
Let $X\subseteq\Nb$ be AP and $X=X_0\cup\ldots X_{r-1}$  be a finite partition. Then $X_i$ is AP for some $i<r$.
\end{Theorem}


It is known that every piecewise syndetic set contains arbitrarily long arithmetic progressions and so Brown's lemma is a generalization of van der Waerden's  theorem. On the other hand, AP sets need not be piecewise syndetic and so partition regularity for piecewise syndetic sets does not imply partition regularity for AP sets.  The proof that piecewise syndetic sets are AP uses the finite version of van der Waerden's theorem (see Rabung \cite{Rab75}).

\begin{Theorem}[Van Der Waerden's Theorem, finite]
For all $r,l$ there exists $n$ such that every coloring $C\colon n\to r$ has a
$C$-homogeneous arithmetic progression of length $l$.
\end{Theorem}
Let  $W(r,l)$ be the least number 
$n$ such that every $r$-coloring of $n$ has a homogeneous arithmetic progression of 
length $l$. By Shelah \cite{She88} the function $W(l,r)$ has primitive recursive upper bounds. These bounds were otained  as a byproduct of a new proof of Hales-Jewett theorem which avoids the use of double induction and can be formalized in 
$\rca$.\footnote{All the previous proofs of Hales-Jewett theorem used double induction giving 
Ackermannian upper bounds. Gowers \cite{Gow98, Gow01} in his 
celebrated work on Szemer\'{e}di's 
theorem obtained elementary recursive upper bounds for the van der Waerden's numbers.
Gowers' bound is the following: 
\[   W(r,l)\leq 2^{2^{f(r,l)}}, \text{ where } f(r,l)=r^{2^{2^{l+9}}}. \]
However, the proof is far from elementary and so we can only conjecture that van der 
Waerden's theorem is provable in $\efa$. On the other hand, the lower bounds for the van der Waerden's numbers are 
exponential (see \cite{GraRotSpe90}), and hence van der Waerden's theorem is 
not provable in bounded arithmetic.}

\begin{Theorem}[Folklore]
The finite version of van der Waerden's theorem is provable in $\rca$.
\end{Theorem}

Working in the framework of reverse mathematics we prove that Brown's 
lemma is equivalent to $\ist$ over $\rca^*$ and so is the statement that piecewise syndetic sets are partition regular. We also show that van der Waerden's 
theorem and partition regularity for AP sets are equivalent to $\bst$ over $\rca^*$.  \smallskip 

We finally consider the finite version of Brown's lemma.  

\begin{Definition}
Let $H\subseteq\Nb$ be finite. Define the \emph{gap size} of $H$, denoted $\gs(H)$, as 
the largest difference between two consecutive elements of $H$. In other words, the gap 
size of $H$ is the least $d$ such that $H$ has gaps bounded by $d$. If $|H|\leq 1$ let $\gs(H)=1$.
\end{Definition}

\begin{Theorem}[Brown's Lemma, finite]
Let $f\colon\Nb\to\Nb$. Then for all $r>0$ there exists $n$ such that every coloring
$C\colon n\to r$ has a $C$-homogeneous set $H$ with $|H|> f(\gs(H))$.
\end{Theorem}

The finite version of Brown's lemma is reminiscent of the Paris-Harrington principle. We can think of $|H|> f(\gs(H))$ as a largeness condition on $H$. 

\begin{Definition}
For $f\colon\Nb\to\Nb$, let $\bl_f$ be the statement ``For all $r>0$ there exists $n$ 
such that every 
$C\colon n\to r$ has a $C$-homogeneous set $H$ such that $|H|> f(\gs(H))$.
\end{Definition}

We show that the finite version of Brown's lemma  is provable in $\rca$ but not in $\rca^*$.  We obtain the latter result by showing that $\bl_f$ is not provable in $\efa$ for  $f(d)=2^d$. On the other hand, $\bl_f$ is provable in $\efa$ for $f(d)=d$.  Therefore  the finite version of Brown's lemma provides a natural example of phase transition with respect to  $\efa$. \smallskip

In the appendix we discuss the original proof of Brown's lemma and point out a disguised use of $\aca$.

\section{Preliminaries}

We assume familiarity with the basic systems 
of reverse mathematics. The standard reference is \cite{Sim09}.   
In this paper we are mainly concerned with the base systems $\rca^*$, $\rca$ and the induction schemes $\ist$ and $\bst$.  Recall that $\rca^*$ consists of the usual axioms for addition, multiplication, exponentiation plus $\Sigma^0_0$-induction and $\Delta^0_1$-comprehension. The system $\rca$ consists of the usual axioms for addition and multiplication plus $\Sigma^0_1$-induction and $\Delta^0_1$-comprehension. We have $\rca=\rca^*+\mathsf{I}\Sigma^0_1$. We  use the following fact repeatedly.
\begin{Theorem}[see Hirst \cite{Hirst87} and Yokoyama \cite{Yokoyama13}]
Over $\rca^*$, $\rt^1_{<\infty}$ is equivalent to $\bst$.
\end{Theorem}
 
Simpson and Smith \cite{SimSmi86} proved that the first-order part of $\wkl^*$, that is $\rca^*$ plus  Weak K\"{o}nig's Lemma,  is the same as that of $\bs1+\exp$ and that $\wkl^*$ is $\Pi_2$-conservative over $\is0+\exp$, also known as $\efa$ (Elementary Function Arithmetic).  Recall that the provably recursive functions of $\efa$ are exactly the elementary recursive functions and that every elementary recursive function $f(n)$ is dominated by
$2_k(n)$ for some number $k$, where $2_k(n)$ is the superexponential function defined by $2_0(n)=n$ and 
$2_{k+1}(n)=2^{2_k(n)}$ (see for instance \cite{SchWai12}).

\section{Piecewise syndetic sets}

In the present paper we do not consider the characterization of piecewise syndetic sets in terms of $\beta\Nb$ (cf.\ Hindman \cite[Theorem 4.40]{Hind12}).  The purpose of this section is to show that, with  few exceptions, most of the elementary characterizations for piecewise syndetic sets can be proved within $\rca^*$.

\begin{Definition}
An infinite set $X\subseteq\Nb$ is \emph{syndetic} if there exists $d\in\Nb$ such that 
$X$ has gaps bounded by $d$. We say that $X$ is $d$-syndetic if $d$ is such a witness. A set $X\subseteq\Nb$ is \emph{thick} if it contains 
arbitrarily long intervals of natural numbers.
\end{Definition}

\begin{Proposition}[$\rca^*$]
Let $X\subseteq\Nb$. The following are equivalent:
\begin{enumerate}[$(1)$]
\item $X$ is piecewise syndetic;
\item $X$ is the intersection of a syndetic set and a thick set.
\end{enumerate}
\end{Proposition}
\begin{proof}
We argue in $\rca^*$. Implication $(2)\imp(1)$ is straightforward.. For $(1)\imp(2)$ suppose $X$ is piecewise $d$-syndetic. Let $Z=\bigcup_{s<d}X+s=\bigcup_{s<d}\{ x+s\colon x\in X\}$. We claim that $Z$ is thick. Fix $n$ and let $H$ be an $n$-element subset of $X$  such that $\gs(H)\leq d$. 
Then $I=\bigcup_{s<d}H+s\subseteq Z$ is an interval of size at least $n$. Define $Y=X\cup (\Nb\setminus Z)$. Clearly $X=Y\cap Z$. It suffices to show that $Y$ is syndetic.  Suppose not. Then $A=\Nb\setminus Y$ is thick and $A\subseteq Z$. In particular $X\cap A=\emptyset$. Let $I\subseteq A$ be an interval of size at least $d$. Every element of $I$ is of the form $x+s$ with $x\in X$ and $s<d$. It follows that $X\cap I\neq\emptyset$ since otherwise $|I|<d$, and so $X\cap A\neq\emptyset$, a contradiction. 
\end{proof}

\begin{Definition}
We say that $I\subseteq X$ is an \emph{interval} of $X$ if for all $x<y$ in $I$ we have $(x,y)\cap X=I$.
\end{Definition}

\begin{Proposition}
Let $X\subseteq\Nb$. The following are equivalent:
\begin{enumerate}[$(1)$]
\item  There exists $d$ such that $\gs(H)\leq d$ for arbitrarily large finite sets $H\subseteq X$, i.e.\ $X$ is piecewise syndetic;
\item There exists $d$ such that $\gs(H)=d$ for arbitrarily large finite sets $H\subseteq X$;
\item There exists $d$ such that $\gs(I)\leq d$ for arbitrarily long intervals $I$ of $X$;
\item There exists $d$ such that $\gs(I)=d$ for arbitrarily long intervals $I$ of $X$.
\end{enumerate}
\end{Proposition}

Most implications are trivial and provable in $\rca^*$. The only nontrivial implications are $(1)\imp(2)$ and $(1)\imp(4)$. 

\begin{Proposition}\label{d} Over $\rca^*$, $\bst$ is equivalent to  $(1)\imp(4)$.
\end{Proposition}
\begin{proof}
We work in $\rca+\bst$ and assume $(1)$. Suppose $X$ is piecewise $d$-syndetic. For all $n$, search for an $n$-element interval $I_n$ of $X$ such that $\gs(I_n)\leq d$. Define a coloring $C\colon \Nb\to d+1$ by letting $C(n)=\gs(I_n)$. By $\rt^1_{<\infty}$ there exists $e\leq d$ such that $\{n\colon C(n)=e\}$ is infinite. Then $e$ is as desired.

We work in $\rca^*$ and assume $(1)\imp(4)$. We aim to show $\rt^1_{<\infty}$. Let $C\colon\Nb\to r$. For convenience we assume $C(n)>0$ for all $n$. 
We define  a piecewise $r$-syndetic set $X=\{x_0<x_1<\ldots\}\subseteq \Nb$ such that for every $n>0$ there is an interval $I_n$ of $X$ of size $n$ with $\gs(I_n)=C(n)< r$ and $\min(I_{n+1})-\max(I_n)=n$.
\[ 1\underbrace{1\overbrace{0\ldots\ldots\ldots0}^{C(2)-1\text{ times}} 1}_{|I_2|=2}0
\underbrace{1\overbrace{0\ldots\ldots\ldots0}^{C(3)-1 \text{ times}}1\overbrace{0\ldots\ldots\ldots0}^{C(3)-1 \text{ times}}1}_{|I_3|=3}001\ldots
\]
Let $x_0=0$, $x_1=x_0+1$ and $x_2=x_1+C(2)$, $x_3=x_2+2$, $x_4=x_3+C(3)$ and $x_5=x_4+C(3)$. In general, suppose we have defined $x_k$ and $k+1=n(n+1)/2=1+2+3+\ldots+n$. Then 
\[ \begin{split} x_{k+1}& =x_k+n; \\ 
               x_{k+i+1}& =x_{k+i}+C(n+1) \text{ for } 0<i<n+1. \end{split}\]
\sloppy Since $x_k\leq rk(k+1)/2$, we can define $x_k$ by bounded primitive recursion and  $X$ by $\Delta^0_1$-comprehension.
By construction, $I_n=\{x_k,x_{k+1},\ldots,x_{k+n-1}\}$, where $k=(n-1)n/2$, is an interval of $X$ of size $n$ and gaps bounded by $r$, and hence $X$ is piecewise $r$-syndetic.  By $(1)\imp(4)$, there exists $d$ such that $\gs(I)=d$ for arbitrarily long intervals $I$ of $X$. We claim that $d<r$ and $\{n\in\Nb\colon C(n)=d\}$ is infinite.  Suppose for a contradiction that $d\geq r$ or there exists $n$ such that $C(m)\neq d$ for all $m>n$. In both cases there exists $n>d$ such that $C(m)\neq d$ for all $m>n$.   Let $k+1=2+3+\ldots+n$. It follows that if $I$ is a $2$-element interval  of $X$ such that $\gs(I)=d$ then  $I\subseteq x_k+1$ and hence $|I|\leq x_k+2$, against our assumption on $d$.
\end{proof}

The reversal in the proof of Theorem \ref{d} does not work for $(1)\imp(2)$. In fact, if the set $\{x\in\Nb\colon C(x)=d\}$ is infinite, then for every multiple $e$ of $d$ there are arbitrarily large subsets $H$ of $X$ with $\gs(H)=e$. 
\begin{Question}
What is the strength of $(1)\imp(2)$? By Theorem \ref{d}, $(1)\imp(2)$ is provable in $\rca+\bst$. 
\end{Question}

\section{Brown's lemma vs van der Waerden's theorem}\label{section vdw}

The proof that every piecewise syndetic set contains arbitrarily long arithmetic progressions is based on the finite version of
van der Waerden's theorem and can be formalized in $\rca^*$.  
Therefore it is not surprising that within $\rca$  Brown's lemma implies van 
der \mbox{Waerden's} theorem. Here we 
establish this fact by showing that van der 
Waerden's theorem is equivalent to $\bst$ over $\rca^*$.

\begin{Theorem}
Over $\rca^*$,  the following are equivalent:
\begin{enumerate}[$(1)$]
 \item $\bst$;
 \item Van der Waerden's theorem;
 \item  Partition regularity of AP sets.
\end{enumerate}
\end{Theorem}
\begin{proof}
We argue in $\rca^*$. Clearly $(3)\imp(2)$. As van der Waerden's theorem implies $\rt^1_{<\infty}$, which is equivalent to 
$\bst$, we have $(2)\imp(1)$. It remains to show $(1)\imp (3)$. Let $X=X_0\cup\ldots X_{r-1}$ be a finite partition of an $AP$ set. Suppose for a contradiction that  $X_i$ is not AP for all $i<r$. Then 
for all $i<r$ there exists $l$ such that no arithmetic 
progression of length $l$ lies within $X_i$. By $\bst$ there exists $l$ 
large enough such that for all $i<r$ no arithmetic progression of length $l$ lies within $X_i$. By the 
finite van der Waerden's theorem, let $n$ be such that every coloring $C\colon n\to r$ 
has a $C$-homogeneous arithmetic progression of length $l$. Let $\{x_0<x_1<\ldots< x_{n-1}\}\subseteq X$ be an arithmetic progression of length $n$. Define a coloring $C\colon n\to r$ be letting $C(m)=i$ iff $x_m\in X_i$. Then there exists a $C$-homogeneous arithmetic progression $m_0<m_1<\ldots m_{l-1}$ of length $l$.  It follows that $x_{m_0}<x_{m_1}<\ldots x_{m_{l-1}}$ is an arithmetic progression of length $l$ which lies entirely within $X_i$ for some $i<r$, for the desired contradiction.  
\end{proof}

Notice that we use the finite van der Waerden's theorem in the course of the proof. 

\begin{Corollary}
Over $\rca^*$, Brown's lemma implies  van der Waerden's theorem.
\end{Corollary}
\begin{proof}
Over $\rca^*$ Brown's lemma implies $\rt^1_{<\infty}$, which is equivalent to $\bst$.
\end{proof}

\section{Brown's lemma}

We first show that Brown's lemma is provable in $\ist$. The argument is due to Kreuzer \cite{Kreuzer12}.

\begin{Lemma}[$\rca^*$]\label{base}
If $X\subseteq \Nb$ is $d$-syndetic and $X=X_0\cup X_1$ is a $2$-partition, then either 
each $X_i$ is piecewise $d$-syndetic or one of them is syndetic.
\end{Lemma}
\begin{proof}
Suppose $X_0$ is not syndetic. This means that for arbitrarily long intervals $I$ of $\Nb$ we have $X_0\cap I=\emptyset$ and so $X\cap I=X_1\cap I$.  Therefore  $X_1$ is piecewise $d$-syndetic.
\end{proof}

\begin{Theorem}\label{forward}
Over $\rca$, $\ist$ implies Brown's Lemma.
\end{Theorem}

\begin{proof}
Let $C\colon\Nb\to r$ be a finite coloring. By bounded $\Sigma^0_2$-comprehension let
\[    I=\{A\subseteq r\colon \bigcup_{i\in A}\{x\in\Nb\colon C(x)= i\} \text{ is 
syndetic}\}.   \]
$I$ is nonempty as $r\in I$. Let $A\in I$ be minimal (w.r.t.\ $\subseteq$). Notice that 
$A\neq\emptyset$.  Suppose that the union is $d$-syndetic. By Lemma \ref{base} and by the minimality of $I$, for 
every $i\in A$ the set $\{x\in\Nb\colon C(x)= i\}$ must be piecewise $d$-syndetic. 
\end{proof}

Note that the proof of Theorem \ref{forward} does not generalize to partitions of piecewise syndetic sets. However we have the following:

\begin{Theorem}\label{regular PS}
Over $\rca^*$, the following are equivalent:
\begin{enumerate}[$(1)$]
 \item Brown's lemma;
 \item Partition regularity of piecewise syndetic sets.
\end{enumerate}
\end{Theorem}
\begin{proof}
We argue in $\rca^*$ and assume Brown's lemma. In particular we have $\Sigma^0_1$-induction. 
Let $X=X_0\cup X_1\ldots\cup X_{r-1}$ be a finite partition of $X$ and suppose that $X$ is piecewise $d$-syndetic.  By primitive recursion we can define a piecewise $d$-syndetic subset $\{x_0<x_1<x_2<\ldots\}$ of $X$ such that for every $n>0$ the $n$-element subset $I_n=\{x_k,x_{k+1},x_{k+2},\ldots,x_{k+n-1}\}\subseteq X$ has gaps bounded by $d$, where $k=1+2+\ldots+(n-1)$. Define $C\colon \Nb\to r$ by letting $C(n)=i$ iff $x_n\in X_i$. By 
Brown's Lemma, there exists a $C$-homogeneous piecewise syndetic set $Y$. Say $Y$ is 
piecewise $e$-syndetic and homogeneous for color $i$. 

We claim  that $Z=\{x_n\colon n\in Y\}\subseteq X_i$ is piecewise $e\cdot d$-syndetic.  To this end it is enough to show that for every $n$ there exists an $n$-element subset $A$ of $Y$ such that $\gs(A)\leq e$ and $\{x_n\colon n\in A\}\subseteq I_m$ for some $m$. Let $n>0$ be given. Set $k=1+2+\ldots+(2ne-1)$ so that $x_k\in B_{2ne}$. Now consider $p=k+2n$.  As $Y$ is piecewise $e$-syndetic, there exists a $p$-element subset $\{i_0<\ldots<i_{p-1}\}$ of $Y$ with gaps bounded by $e$. Let $A=\{i_k<i_{k+1}<\ldots<i_{p-1}\}$. Notice that $|A|=2n$ and $\gs(A)\leq e$. Since  $i_k\geq k$ we have that $x_{i_k}\in I_m$ with $|I_m|\geq 2ne$. As $\gs(A)\leq e$ we have that $i_{p-1}-i_k\leq 2ne$ and so $|\{x_i\colon i_k\leq i\leq i_p\}|\leq 2ne+1$. It follows that  $\{x_i \colon i_k\leq i\leq i_p\}\subseteq I_m\cup I_{m+1}$. This ensures that $\{x_i\colon i\in A\}\subseteq I_m\cup I_{m+1}$. Then either the first $n$ elements are in $I_m$ or the last $n$ elements are in $I_{m+1}$. In both cases we are done.
\end{proof}

We now turn to the reversal. Actually we show that the following weak version of Brown's lemma already implies $\ist$ over $\rca^*$.

\begin{Theorem}[Weak Brown's Lemma]
For every coloring $C\colon\Nb\to r$ there exists $d\in\Nb$ such that  $\gs(H)\leq d$ 
for arbitrarily large $C$-homogeneous sets $H$.
\end{Theorem}

\begin{Theorem}\label{reversal} Over $\rca^*$, the following are equivalent:
\begin{enumerate}[$(1)$]
 \item $\ist$;
 \item Brown's lemma;
 \item Partition regularity of piecewise syndetic sets;
 \item Weak Brown's lemma.
\end{enumerate}
\end{Theorem}

We need the following diagonalization lemma. 

\begin{Lemma}[$\rca^*$]\label{diagonalization}
There exists a function $D\colon\Nb\times\Nb\to 2$ such that 
for all 
$d$ the $2$-coloring $D(d,\cdot)$ has no arbitrarily large $C$-homogeneous 
sets $H$ with $\gs(H)\leq d$.
\end{Lemma}
\begin{proof}
For all $d>0$ let  
\[   D(d,\cdot)=\overbrace{0\ldots 0}^{d \text{ times}}\overbrace{1\ldots 1}^{d \text{ times}}\overbrace{0\ldots 0}^{d \text{ times}}\overbrace{1\ldots 1}^{d \text{ times}}\ldots  \]
We define $D$ by $D(d,x)=\floor{\frac{x}{d}}$ (mod $2$).  Fix $d>0$ and let $C(x)=D(d,x)$. Suppose $H$ is 
$C$-homogeneous with gaps bounded by $d$. We claim that  $|H|\leq d$.  Let
$H=\{x_0<x_1<\ldots < x_l\}$ and $m=\floor{\frac{x_0}{d}}$ so that $md\leq x_0<(m+1)d$. We show by induction that $x_l<(m+1)d$ for all $l$.
The case $l=0$ is true. Suppose $x_l<(m+1)d$.  Since $x_{l+1}-x_l\leq d$ we have that $x_{l+1}<(m+2)d$. If $(m+1)d\leq x_{l+1}$ then $x_{l+1}/d=m+1$ and $C(x_{l+1})=m+1$ (mod $2$) $\neq m$ (mod $2$) $=C(x_l)$, against $H$ being homogeneous. Therefore $H\subseteq [md,(m+1)d)$ and hence $|H|\leq d$.
\end{proof}

\begin{proof}[Proof of Theorem \ref{reversal}]
 Implication $(1)\imp(2)$  is Theorem \ref{forward} and $(2)\biimp(3)$ is Theorem \ref{regular PS}.
Clearly $(2)\imp(3)$. It remains to show $(3)\imp (1)$. 

Within $\mathsf{I}\Sigma^0_{0}$, one can show that $\mathsf{I}\Sigma^0_{n+1}$ is equivalent to $\mathsf{S}\Pi^0_n$, strong collection for $\Pi^0_n$-formulas (see \cite[Ch.\ I Sect.\ 2(b)]{HajPud93}). We first assume $\mathsf{I}\Sigma^0_{1}$ and  prove 
$\mathsf{S}\Pi^0_1$, that is:
\[   (\forall a)(\exists d)(\forall x<a)\big(\exists y\forall z\theta(x,y,z)\rightarrow (\exists 
y<d)(\forall z)\theta(x,y,z)\big), \]
where $\theta$ is $\Sigma^0_0$. Let $f(x,s)$ be the 
least $y\leq s$ such that $(\forall z<s)\theta(x,y,z)$, and $s$ if there is no such a $y$. By $\Sigma^0_1$-induction one can show that 
if $(\exists y)(\forall z)\theta(x,y,z)$  then $f(x)=\lim_{s\to\infty}f(x,s)=$ the 
least $y$ such that $(\forall z)\theta(x,y,z)$.  

Define a coloring $C\colon \Nb\to 2^a$ as follows. Let 
$D\colon\Nb\times\Nb\to 2$ be the function of Lemma \ref{diagonalization} and set 
$C(s)=\langle D(f(x,s),s)\colon x<a\rangle$. By $(3)$ there exists $d$ such that  $\gs(H)\leq d$ 
for arbitrarily large $C$-homogeneous sets $H$. We claim that $d$ is as desired.  
Let $x<a$  and suppose that $f(x)$ exists. Fix $s$ such that $f(x,t)=f(x)$ for all $t>s$. We aim to prove that $D(f(x),\cdot)$ has arbritrarily large homogeneous sets $H$ with $\gs(H)\leq d$ and therefore it must be $f(x)<d$. Let $k$ be given. By the assumption, there exists a $C$-homogeneous set $G$ of size $s+k+1$ and $\gs(G)\leq d$. Let $H$ consist of the last $k$ elements of $G$. Then $|H|=k$, $\gs(H)\leq d$ and $s<t$ for all $t\in H$. Suppose $G$ is $C$-homogeneous for color $i$. Then for all $t\in H$ we have that $i(x)=C(t)(x)=D(f(x,t),t)=D(f(x),t)$ and hence $H$ is homogeneous for $D(f(x),\cdot)$.

It remains to show that $(3)$ implies $\mathsf{I}\Sigma^0_1$ over $\rca^*$. We can apply the argument above to prove $\mathsf{S}\Pi^0_1$, that is:
\[   (\forall a)(\exists d)(\forall x<a)\big(\exists y\theta(x,y,z)\rightarrow (\exists 
y<d)\theta(x,y,z)\big), \]
where $\theta$ is $\Sigma^0_0$. Define $f(x,s)$ to be the least $y<s$ such that $\theta(x,y)$ if $y$ exists and $s$ otherwise. Now $\Sigma^0_0$-induction is sufficient to show that if $\exists y\theta(x,y)$ then $f(x)=\lim_{s\to\infty}f(x,s)=$ the 
least $y$ such that $\theta(x,y)$. Define $C$ as before by using $D$ and show that $d$ is as desired.  
\end{proof}

\section{Finite Brown's Lemma}

\begin{Theorem}[Brown's Lemma, finite]
Let $f\colon\Nb\to\Nb$. Then for all $r>0$ there exists $n$ such that every coloring
$C\colon n\to r$ has a $C$-homogeneous set $H$ with $|H|> f(\gs(H))$.
\end{Theorem}
\begin{proof}
From Brown's Lemma by using K\"{o}nig's Lemma.
\end{proof}

\begin{Theorem}\label{finite brown}
The finite version of Brown's Lemma is provable in $\rca$.
\end{Theorem}
\begin{proof}
The proof of  \cite[Theorem 10.33]{LanRob04} can be formalized in $\rca$. Let 
$f\colon\Nb\to\Nb$ be given. Without loss of generality we may assume that $f$ is nondecreasing, for otherwise we can define $g(n)=\sum_{i\leq n}f(i)$. The proof is by  
$\Sigma^0_1$-induction on $r$.  

For $r=1$, let $n_1=f(1)+2$.  Suppose $n_r$ works for $r$. We claim that $n_{r+1}=(r+1)f(n_r)+1$ works for 
$r+1$. Let $C\colon n_{r+1}\to r+1$ be any coloring. Set $H_i=\{x<n_{r+1}\colon C(x)= i\}$ for 
$i\leq r$. We may assume that $|H_i|\leq f(\gs(H_i))$ for every $i\leq r$, for otherwise we are 
done. We may also assume that $\gs(H_i)\leq n_r$ for every $i\leq r$, for otherwise 
there exists an interval of size $n_r$ which avoids some color $i$, and the conclusion 
follows by induction. As $f$ is nondecreasing, we have that 
$f(\gs(H_i))\leq f(n_r)$ for all $i\leq r$. Therefore, $|H_i|\leq f(n_r)$ for every 
$i\leq r$, and  
\[        n_{r+1}=\sum_{i\leq r}|H_i|\leq (r+1)f(n_r), \] a contradiction. 
\end{proof}

We define Brown's numbers as follows. 

\begin{Definition}
For $f\colon\Nb\to\Nb$ and $r>0$ let $B_f(r)$ be the least natural number $n$ 
such that every $r$-coloring of $n$ has a homogeneuos set $H$ with  
$|H|>f(\gs(H))$.  
\end{Definition}

Remember that the superexponential function $2_k(n)$ is defined by $2_0(n)=n$ and $2_{k+1}(n)=2^{2_k(n)}$. Let $2_r=2_r(1)$.
The proof of Theorem \ref{finite brown} gives superexponential upper 
bounds $n_r$ for $B_f(r)$.  For instance, if $f(d)=2^d$, then $n_r\geq 2_r$.

\begin{Theorem}[Ardal \cite{Ardal10}]
Let $f(d)=md$ with $m>0$. Then for all $r>0$, $B_f(r)\leq r(2^{mr}-mr)+1$.
\end{Theorem}

Recall that, for $f\colon\Nb\to\Nb$, $\bl_f$ is the statement ``For all $r>0$ there exists $n$ 
such that every  $C\colon n\to r$ has a $C$-homogeneous set $H$ such that $|H|> f(\gs(H))$. 

\begin{Theorem}
$\efa\vdash (\forall m>0)\bl_{d\mapsto md}$.
\end{Theorem}
\begin{proof}
The reader can check that the proofs of Lemma 10, 11, 12, 13, 14, 22, 23 and Theorem 15, 24 of \cite{Ardal10} can be formalized in $\efa$.
\end{proof}

We aim to show that  $\rca^*$ does not prove the finite version of Brown's lemma by proving that $B_f(r)$ is not elementary recursive for $f(d)=2^d$.

\begin{Proposition}
The function $g(n)=2_{\log_2 n}$ is not elementary recursive.
\end{Proposition}
\begin{proof}
Suppose $g$ is elementary recursive. Then there exists a number $k$ such that $2_{\log_2 n}<2_k(n)$ for all $n$.
Let $n=2_{k+2}$. Then $2_{\log_2 n}=2_{2_{k+1}}$ and $2_k(n)=2_{2k+2}$. On the other hand $2_{k+1}\geq 2k+2$ for all $k$, and hence $2_{2_{k+1}}\geq 2_{2k+2}$, a contradiction. 
\end{proof}

\begin{Theorem}\label{not elementary}
Let $f(d)=2^d$. Then for all $r>0$, $B_f(r)> 2_{\log_2(r)}$. 
\end{Theorem}
\begin{proof}
 For all $s\geq 0$ we define a number $n_s\geq 2_s$ and a coloring $C_s\colon n_s\to 2^s$ witnessing $B(2^s)>n_s$. \smallskip

In general, for $f\colon\Nb\to\Nb$ nondecreasing, $B_f(r)>n$ iff there exists a  coloring $C\colon n\to r$ such that for all $i<r$ the finite set $H=\{x<n\colon C(x)=i\}$ satisfies 
\[ \tag{$*$} |I|\leq f(\gs(I)) \text{ for every interval $I$ of $H$}. \] 
Notice that $(*)$ is shift invariant.  For finite sets $A<B$ let $d(A,B)=\min B-\max A$. Note that if $A$ satisfies $(*)$ and $m\cdot|A|\leq f(d)$, then the set $H=\bigcup_{l<m} A_l$ satisfies $(*)$, where $A_0<A_1<\ldots< A_{m-1}$, each $A_l$ is a shift of $A$ and $d(A_l,A_{l+1})=d$.\smallskip

We ensure that for all $s$ and for all $i<2^s$ the set $\{x<n_s\colon C_s(x)=i\}$ satisfies $(*)$ with $f(d)=2^d$.
For $s=0$, let $n_0=2$ and $C_0=00$. For $s=1$, let $n_1=2^4=16$ and 
\[      C_1=0011001100110011. \]
Suppose we have defined $n_s$ and $C_s\colon n_s\to 2^s$. Let $n_{s+1}=2^{s+1}\cdot 2^{n_0+n_1+\ldots+n_s+1}$.
Define $C_{s+1}\colon n_{s+1}\to 2^{s+1}$ by 
\[   C_{s+1}=C_sD_sC_s D_s\ldots C_sD_s, \]
where $C_sD_s$ is repeated $2^{n_s}$ times and $D_s$ is a copy of $C_s$ with the color $i<2^s$ replaced by $i+2^s$.
By induction one can show that $n_{s+1}=2n_s2^{n_s}$. 
\begin{Claim}
For all $s$ and for all $i<2^s$ the set 
$H_{i,s}=\{x<n_s\colon C(x)=i\}$ satisfies $(*)$.
\end{Claim}
By induction on $s$ we prove that for all $i<2^s$:
\begin{enumerate}[$(1)$]
 \item $|H_{i,s}|=n_s/2^s$;
 \item $H_{i,s}$ satisfies $(*)$;
 \item $n_s=\max H_{i,s}-\min H_{i,s}+n_0+\ldots+ n_{s-1}+1$.
\end{enumerate}
For $s=0$ this is true. Suppose it is true for $s$ and let $m=2^{n_s}$. Fix $i<2^{s+1}$.  By construction either $i<2^s$ and $H_{i,s+1}=\bigcup_{l<m} A_l$, where $A_0<A_1<\ldots <A_{m-1}$ and $A_l=H_{i,s}+2ln_s$, or $i=j+2^s$ with $j<2^s$ and
$H_{i,s+1}=\bigcup_{l<m} A_l$, where $A_0<A_1<\ldots<A_{m-1}$ and $A_l=H_{j,s}+(2l+1)n_s$. Suppose $i<2^s$. The argument for $i=j+2^s$ is similar. Then
\[
  |H_{i,s+1}|=m\cdot |H_{i,s}|\stackrel{(1)}{=} 
      m\cdot \frac{n_s}{2^s}=2^{n_s}\cdot\frac{n_s}{2^s}=
      \frac{n_{s+1}}{2^{s+1}}.
  \]
We aim to show that $H_{i,s+1}$ satisfies $(*)$. By induction  $H_{i,s}$ satisfies $(*)$ and so every $A_l$ satisfies $(*)$ because $(*)$ is invariant under shift. Since $m\cdot|A_0|=|H_{i,s+1}|=2^{n_s}\cdot\frac{n_{s}}{2^s}=2^{n_0+\ldots+n_s+1}$, it is sufficient to show that $d(A_l,A_{l+1})=d(A_0,A_1)=n_0+n_1+\ldots+ n_{s}+1$.  Now 
\begin{multline*}
d(A_0,A_1)= n_s+\min H_{i,s}+(n_s-\max H_{i,s})\stackrel{(3)}{=}\\ n_s+(n_0+\ldots+n_{s-1}+1)=n_0+\ldots+n_s+1.
\end{multline*}
Let $a=\min H_{i,s+1}$ and $b=\max H_{i,s+1}$. It remains to show that $n_{s+1}=b-a+n_0+\ldots+ n_{s}+1$. Write 
$n_{s+1}$ as $a+(b-a)+(n_{s+1}-b)$.
Now 
\[   a=\min H_{i,s+1}=\min H_{i,s} \]
and
\[  n_{s+1}-b=n_{s+1}-\max H_{i,s+1}= (n_s-\max H_{i,s})+n_s. \]
It follows that 
\begin{multline*}  n_{s+1}=\min H_{i,s}+(b-a)+n_s-\max H_{i,s}+n_s =  \\
     b-a + n_s +(n_s+\min H_{i,s}-\max H_{i,s}) \stackrel{(3)}{=}  
     b-a +n_s+(n_0+\ldots +n_{s-1}+1) = \\ b-a+n_0+\ldots n_s+1. \end{multline*}
This completes the proof of the claim.  By induction it is easy to prove that $n_s\geq 2_s$. It follows that 
\[ B(r)\geq B(2^{\log_2(r)})> n_{\log_2(r)}\geq 2_{\log_2(r)}, \] as desired.
\end{proof}

Notice that Theorem \ref{not elementary} is provable in $\is1$.

\begin{Corollary}
$\efa\nvdash\bl_{d\mapsto 2^d}$. In particular, $\rca^*$ does not prove the finite version of Brown's lemma.
\end{Corollary}
\begin{proof}
Let $f(d)=2^d$. The function $B_f(r)$ is $\Sigma_0$-definable and so $\efa\vdash  \bl_{f}$ iff $B_f(r)$ is provably recursive in $\efa$ iff $B_f(r)$ is elementary recursive. By Theorem \ref{not elementary} the function $B_f(r)$ is not elementary recursive. Therefore $\efa\nvdash\bl_f$. The second part follows from the fact that the statement $\bl_{f}$ is $\Pi_2$ and $\rca^*$ is conservative over $\efa$ for $\Pi_2$-sentences. 
\end{proof}

As the finite van der Waerden's theorem is already 
provable in $\rca$ and presumably in $\rca^*$, the question whether Brown's lemma 
implies van der Waerden's theorem can be settled only over a very weak system of 
arithmetic.

\begin{Question}
What is the relationship between (the finite versions of) Brown's lemma and van der 
Waerden's theorem over a suitable bounded fragment of second-order arithmetic?
\end{Question}

\appendix

\section{}

The classical proof of Brown's lemma  (see for instance
\cite[Lemma 1]{Brown71} or \cite[Theorem 10.32]{LanRob04}) is 
based on the following fact. 

\begin{Fact}\label{fact}
Let $r\in\Nb$. If $S\subseteq r^{<\Nb}$ is infinite, then there exists 
$g\colon\Nb\to r$ such that for all $k$ there is $\sigma\in S$ with 
$g\restriction k\subseteq\sigma$.
\end{Fact}

\begin{proof}[\textbf{Classical proof of Brown's lemma}]
By induction on $r$. If $r=1$, there is nothing to prove. Let $C\colon \Nb\to r+1$. We 
say that $\sigma\in \Nb^{<\Nb}$ is a factor of $C$, and write $\sigma\in S(C)$, if 
$\sigma\subseteq \langle C(y),C(y+1),C(y+2),\ldots\rangle$ for some 
$y$. Let $S$ be the set of factors which avoid the color $r$, that is 
$S=S(C)\cap r^{<\Nb}$. 

We may assume that $\{x\in\Nb\colon C(x)=r\}$ is not piecewise syndetic, otherwise we are done. We claim 
that for every $d$ there is $\sigma\in S$ of length $d$. Suppose not. Let $d$ be 
such that every $\sigma\in S(f)$ of length $d$ does not avoid $r$. Then it is easy 
to see that $\{x\in\Nb\colon C(x)=r\}$ is piecewise $d$-syndetic, contrary to our assumption. Therefore 
$S$ is infinite.  

By Fact \ref{fact}, there exists $g\colon \Nb\to r$ such that for all 
$n$ there exists 
$\sigma\in S$ with $g\restriction n\subseteq \sigma$. By induction, let $i<r$ such 
that $\{x\in\Nb\colon g(x)=i\}$ is piecewise syndetic and let $d$ be a witness.  We claim that 
$\{x\in\Nb\colon C(x)=i\}$ is piecewise $d$-syndetic.  Fix $n$. There exists a set 
$F$ of size $n$ and gaps bounded by $d$ such that $g(x)=i$ for all $x\in F$. Let $k>F$. 
By the assumption on $g$, let $\sigma\in S$ such that $g\restriction k\subseteq\sigma$. 
Now, 
$\sigma$ is a factor of $C$, say $\sigma \subseteq \langle C(y),C(y+1),\ldots\rangle$. 
Hence, for all $x<k$ we have that $g(x)=C(y+x)$. It easily follows that $y+F$ is as 
desired (size $n$, gaps bounded by $d$, $C$-homogeneous for color $i$).
\end{proof}

A loose formalization of the above proof goes through 
$\aca$ plus $\Pi^1_1$-induction.  In particular, Fact \ref{fact} is provable in $\aca$ 
and  actually equivalent to $\aca$. 

\begin{Proposition}\label{aca}
Over $\rca$, the following are equivalent:
\begin{enumerate}[ $(1)$]
 \item $\aca$
 \item Let $r\in\Nb$. If $S\subseteq r^{<\Nb}$ is infinite, then there exists 
$g\colon\Nb\to r$ such that for all $n$ there is $\sigma\in S$ with $g\restriction 
n\subseteq\sigma$.
\item If $S\subseteq 2^{<\Nb}$ is infinite, then there exists $g\colon\Nb\to 2$ such that 
for all $n$ there is $\sigma\in S$ with $g\restriction n\subseteq\sigma$.
\end{enumerate}
\end{Proposition}
\begin{proof}
$(1)$ implies $(2)$. Define a tree $T\subseteq r^{<\Nb}$ by letting $\tau\in T$ 
iff $(\exists \sigma\in S)\tau\subseteq\sigma$. $T$ is an infinite finitely branching 
tree. By K\"{o}nig's Lemma, there exists an infinite path $g$. Then $g$ is as desired.
$(2)$ implies $(3)$ is trivial. 

For $(3)$ implies $(2)$, let $f\colon\Nb\to\Nb$ be a 
one-to-one 
function. We aim to show that the range of $f$ exists. Define $S\subseteq 2^{<\Nb}$ by 
letting $\sigma\in S$ if and only if $(\forall y<|\sigma|)(\sigma(y)=1\leftrightarrow 
(\exists x<|\sigma|)f(x)=y)$. Then $S$ is infinite. By $(2)$, let $g$ be such that every 
initial segment of $g$ is an initial segment of some $\sigma\in S$. We claim that 
$y\in\ran f$ iff $g(y)=1$. Suppose $g(y)=1$ and let $g\restriction y+1\subseteq\sigma$ 
with $\sigma\in S$. Then $\sigma(y)=1$ and so $y\in\ran f$. Suppose that $y\in\ran f$ and 
let $f(x)=y$. Let $n>x,y$ and $\sigma\in S$ such that $g\restriction n\subseteq\sigma$. 
Then $\sigma(y)=1$ and so $g(y)=1$.
\end{proof}

\bibliography{brown}{} 
\bibliographystyle{alpha}
 
\end{document}